\documentclass[11pt]{article}
\usepackage[a4paper,hmargin={2.5cm,2.5cm},vmargin={2.5cm,2.5cm}]{geometry}

\usepackage[utf8]{inputenc}
\usepackage[T1]{fontenc}
\usepackage{lmodern}
\usepackage[english]{babel}
\usepackage{mathrsfs}
\usepackage{amsmath,amssymb,amsthm}
\usepackage{hyperref}
\usepackage[shortlabels]{enumitem}
\usepackage{csquotes}
\usepackage{tikz-cd}

\theoremstyle{plain}
\newtheorem{theorem}{Theorem}[section]

\newtheorem{proposition}[theorem]{Proposition}
\newtheorem{corollary}[theorem]{Corollary}

\theoremstyle{definition}
\newtheorem{definition}[theorem]{Definition}
\newtheorem{example}{Example}[section]

\theoremstyle{remark}
\newtheorem{remark}[theorem]{Remark}

\numberwithin{equation}{section}

\newcommand{\N}{\mathbb{N}}

\newcommand{\R}{\mathbb{R}}

\newcommand{\wto}{\rightharpoonup}

\tikzset{
    symbol/.style={
        draw=none,
        every to/.append style={
            edge node={node [sloped, allow upside down, auto=false]{$#1$}}}
    }
}

\title{\textbf{Quasilinear Equations with Neumann Boundary Conditions}}

\author{
\textsc{Annamaria Canino}\ (\textit{corresponding author})\\
\small Dipartimento di Matematica e Informatica, Università della Calabria\\
\small Ponte Pietro Bucci, cubo 31B, 87036 Arcavacata di Rende, Cosenza, Italy\\
\small Email: \texttt{annamaria.canino@unical.it}
\and
\textsc{Simone Mauro}\\
\small Dipartimento di Matematica e Informatica, Università della Calabria\\
\small Ponte Pietro Bucci, cubo 31B, 87036 Arcavacata di Rende, Cosenza, Italy\\
\small Email: \texttt{simone.mauro@unical.it}
}

\date{\today}

\begin{document}

\maketitle

\begin{abstract}
We prove a multiplicity result for non-constant weak solutions $u \in H^1(\Omega)$ for the quasilinear elliptic equation
\[
\begin{cases}
-\mathrm{div}(A(x,u)\nabla u) + \dfrac{1}{2} D_s A(x,u)\,\nabla u \cdot \nabla u = g(x,u) - \lambda u & \text{in } \Omega,\\[4pt]
A(x,u)\nabla u \cdot \eta = 0 & \text{on } \partial \Omega,
\end{cases}
\]
where $\lambda \in \mathbb{R}$, $\Omega$ is a bounded Lipschitz domain, $\eta$ is the outward normal to $\partial\Omega$, and $g(x,u)$ is a Carathéodory function satisfying a general subcritical and superlinear growth condition. 
We also prove that any weak solution is bounded under a stronger growth assumption.
\end{abstract}

\noindent\textbf{Keywords:} Subcritical nonlinearities, gradient elliptic systems, Neumann boundary conditions, quasilinear elliptic equations, nonsmooth critical point theory.

\medskip
\noindent\textbf{2020 MSC:} 35A01, 35A15, 35J05, 35J20, 35J25, 35J62.

\section{Introduction}

Let $N\ge2$, let \( \Omega \subset \mathbb{R}^N \) be a bounded domain with a Lipschitz boundary and let $p\in(2,2^*)$ where
\[
2^*:=
\begin{cases}
    \frac{2N}{N-2}&N\ge3\\
    \infty&N=1,2
\end{cases}
\]
is the critical Sobolev exponent. We denote with $(2^*)'$ the conjugate exponent of $2^*$, which satisfies $\frac{1}{2^*}+\frac{1}{(2^*)'}=1$.
We consider the following elliptic problem:

\[\tag{$\mathcal P$}\label{P}
\begin{cases}
    -\text{div}(A(x,u)\nabla u)+\frac12 D_sA(x,u)\nabla u\cdot\nabla u=g(x,u)-\lambda u&in\ \Omega\\
    A(x,u)\nabla u\cdot\eta=0&on\ \partial\Omega,
\end{cases}
\]
where \( \eta \) is the outward normal at the boundary \( \partial \Omega \) and $\lambda\in\R$.


We assume that $A(x,s)$ is a symmetric matrix $N\times N$ with coefficients $a_{ij}:\Omega\times\R\to\R$ such that $a_{ij}$ are $C^1$-Carathéodory functions. Suppose also that there exist $C,\nu,R>0$ such that
\begin{align}
\tag{$a.1$}\label{a.1}
    &|a_{ij}(x,s)|,\ |D_sa_{ij}(x,s)|\le C,\quad i,j=1,\dots,N,\\
   \tag{$a.2$} \label{a.2}
    &A(x,s)\xi\cdot\xi\ge\nu|\xi|^2,\\
   \tag{$a.3$} \label{a.3}
    &|s|\ge R\implies sD_sA(x,s)\xi\cdot\xi\ge0,
\end{align}
for a.e. $x\in\Omega$, for every $s\in\R$ and $\xi\in\R^N$, where $D_sA(x,s)$ denotes the matrix with coefficients $D_sa_{ij}(x,s)$.

\begin{example}
    Let $A(x,s)=\mathcal A(s)\text{Id}$, where $\text{Id}$ is the identity matrix and 
    \[\mathcal A(s)=1+\varepsilon\arctan(s^2),\quad \varepsilon>0.\]
    We compute the derivative:
    \[\mathcal A'(s)=\varepsilon\frac{2s}{1+s^4}.\]
    Since $\arctan(s^2)\le\frac{\pi}{2}$ and $\mathcal A'(s)\to0$ as $|s|\to+\infty$, we deduce that \eqref{a.1} holds. Clearly, $\mathcal A(s)\ge 1$ and \eqref{a.2} is satisfied with $\nu=1$. Moreover, $s\mathcal A'(s)\ge0$ for every $s$, namely \eqref{a.3} is true. Finally, we note also that
    \[\frac{s\mathcal A'(s)}{\mathcal A(s)}=\varepsilon\frac{2s^2}{1+s^4}\cdot\frac{1}{1+\arctan(s^2)}\le\varepsilon\max_{s\in\R}\frac{2s^2}{1+s^4}=:M\varepsilon.\]
    For $\varepsilon>0$ small enough, i.e. $\varepsilon<\frac{\gamma}{M}$, we have
    \[sD_sA(x,s)\xi\cdot\xi\le \gamma A(x,s)\xi\cdot\xi,\quad \gamma\in(0,p-2),\]
    for every $(x,s,\xi)\in\Omega\times\R\times\R^N$, namely the hypothesis \eqref{a.4} below is also satisfied.
\end{example}

\begin{example}
   Let $A(x,s)=\mathcal A(s)\text{Id}$, where $\text{Id}$ is the identity matrix and 
    \[\mathcal A(s)=1+\frac{|s|^\alpha}{1+|s|^\alpha},\quad \alpha>0.\]
 First of all, we compute $\mathcal A'(s)$:
 \[\mathcal A'(s)=\frac{\alpha|s|^{\alpha-2}}{(1+|s|^{\alpha})^2}s.\]
 We note that 
 \[\lim_{|s|\to+\infty} \mathcal A(s)=2,\quad \lim_{|s|\to+\infty}\mathcal A'(s)=0.\]
 Then \eqref{a.1} and \eqref{a.2} is satisfied for some constant $C\ge2$ and with $\nu=1$. We note also that 
 \[s\mathcal A'(s)=\frac{\alpha|s|^\alpha}{(1+|s|^\alpha)^2}\ge0,\]
 and \eqref{a.3} is satisfied for all $s\in\R$. Moreover, 
 \[\frac{s\mathcal A'(s)}{\mathcal A(s)}=\frac{\alpha|s|^\alpha}{(1+|s|^\alpha)^2}\cdot\frac{1+|s|^\alpha}{1+2|s|^\alpha}\le \frac{\alpha|s|^\alpha}{2|s|^{3\alpha}}\le\frac{\alpha}{2R^{2\alpha}}=:\gamma,\]
 when $|s|\ge R$. In particular, $A(x,s)$ satisfies also:
 \[|s|\ge R\implies sD_sA(x,s)\xi\cdot\xi\le\gamma A(x,s)\xi\cdot\xi,\]
 for every $(x,\xi)\in\Omega\times\R^N$ and with $\gamma\in(0,p-2)$ for $R>0$ large enough, i.e. \eqref{a.4} (see below) holds.
\end{example}

Let $g:\Omega\times\mathbb{R}\to\mathbb{R}$ be a Carathéodory function. 
Assume that there exist constants $q>2$, $b\in\R$ and a function 
$a\in L^{r}(\Omega)$,  with $r\ge \frac{2N}{N+2}$, such that, 
for a.e.\ $x\in\Omega$ and every $s\in\mathbb{R}$,
\begin{align}
\tag{$g.1$}\label{g.1}
    |g(x,s)|
        &\le a(x)+b\,|s|^{p-1},\\[4pt]
\tag{$g.2$}\label{g.2}
    |s|\ge R 
        &\;\Longrightarrow\;
        0<q\,G(x,s)\le s\,g(x,s),
\end{align}
where 
\[
    G(x,s):=\int_{0}^{s} g(x,t)\,dt .
\]

We suppose also a superlinear assumption on the coefficients $A(x,s)$:
\begin{equation}
 \tag{$a.4$}\label{a.4} 
 |s|\ge R\implies sD_sA(x,s)\xi\cdot\xi\le \gamma A(x,s)\xi\cdot\xi,
\end{equation}
for a.e. $x\in \Omega$ and for every $\xi\in\R^N$.

\begin{definition}
We  say that $u\in H^1(\Omega)$ is a weak solution of \eqref{P} if 
\begin{equation*}
\int_\Omega A(x,u)\nabla u\cdot\nabla v+\frac12\int_\Omega \left(D_sA(x,u)\nabla u\cdot\nabla u\right)v=\int_\Omega (g(x,u)-\lambda u)v,\ \ \forall\ v\in C^\infty(\overline\Omega).
\end{equation*}
\end{definition}

Weak solutions of \eqref{P} are formally critical points of the energy functional
$f_\lambda: H^1(\Omega)\to\R$,
 \[f_\lambda(u)=\frac12\int_\Omega A(x,u)\nabla u\cdot\nabla u+\frac{\lambda}{2}\int_\Omega u^2-\int_\Omega G(x,u).\]
 The energy functional  is continuous but differentiable only along directions given by \( v \in H^1(\Omega) \cap L^\infty(\Omega) \) and we can not apply standard variational methods.
This is due to the fact that  
\begin{equation}\label{nonsmooth condition}
D_s A(x, u) \nabla u \cdot \nabla u \notin (H^1(\Omega))',
\end{equation} 
since this quantity generally belongs to the Lebesgue space \( L^1(\Omega) \), which is not contained in \( (H^1(\Omega))' \) for \( N \geq 2 \). This issue can be resolved using the abstract concept of weak slope, defined for continuous functions on a metric space, as developed in \cite{corvellec1993deformation,degiovanni1994critical}. This tool allows us to study quasilinear equations such as \eqref{P}. 

The Dirichlet case has been extensively studied in \cite{caninoquasilineare1, caninoserdica, nonsmooththeory1}, where the authors applied the quoted non-smooth critical point theory to prove the existence of \emph{infinitely many bounded solutions}. 

A quasilinear equation, still in the Dirichlet setting, was also investigated by Arcoya and Boccardo~\cite{arcoyaboccardo1}, who constructed a single bounded solution by means of a~priori estimates.

Later, Candela and Palmieri, in \cite{candela2006multiple, candela2009some, candela2009infinitely}, obtained infinitely many bounded solutions in the case of this equation and a corresponding generalization, developing a differentiable critical point framework for such functionals, working in the space \( W_0^{1,p} \cap L^\infty \) where the functional becomes smooth.

We prove the following existence result for a quasilinear elliptic equation with Neumann boundary conditions, representing (to the best of our knowledge) the first result in this framework.
\begin{theorem}\label{main result1}
Let $\lambda\in\R$ and assume that \eqref{a.1}-\eqref{a.4} and \eqref{g.1}-\eqref{g.2} hold. If 
\[
A(x,-s)=A(x,s)\quad \text{and}\quad g(x,-s)=-g(x,s),
\]
 then there exists a sequence $\{u_h\}\subset H^1(\Omega)$ of non-constant weak solutions of \eqref{P}.

Furthermore,  if \eqref{g.1}  holds with $r>\frac{N}{2}$, then any weak solution $u_h\in H^1(\Omega)$ is bounded, i.e. $u_h\in H^1(\Omega)\cap L^\infty(\Omega)$.
\end{theorem}

{ In particular, if $A(x,s)=\text{Id}$ is the identity matrix, we obtain the elliptic problem:

\[
\begin{cases}
-\Delta u+\lambda u=g(x,u) & \text{in } \Omega, \\
\frac{\partial u}{\partial \eta}=0 & \text{on } \partial \Omega.
\end{cases}
\]
This problem has been extensively studied; see, for example, \cite{lin1988large}, as well as \cite{pariniwethsublinear, saldana2022least} for the case $g(x,s)=|s|^{p-2}s$. The critical problem was initially investigated by \cite{comte1991existence} using the dual method. However, this approach fails in the quasilinear case because, in general, the inverse of the quasilinear operator is not well-defined.

If $g(x,s)=|s|^{p-2}s$, the equation has several physical applications. For instance, it arises in nonlinear optics, where it is related to the Schr\"odinger equation:

\[
\tag{$\mathcal S$}\label{S}
    -i\frac{\partial \Psi}{\partial t}-\Delta\Psi=|\Psi|^{p-2}\Psi \qquad \text{in } \Omega\times(0,+\infty). \\
\]

In this context, the complex-valued function $\Psi(t,x)=e^{i\lambda t}u(x)$ represents a soliton solution of \eqref{S}, meaning that for every fixed $t>0$, the profile remains unchanged. This equation is also relevant in the study of Bose-Einstein condensates, where for $p=4$, it corresponds to the so-called Gross-Pitaevskii equation, which describes the wave function of a condensate at extremely low temperatures \cite{dalfovo1999theory, pitaevskii2016bose}. 

}

\section{Preliminaries}

We recall some results about the critical point theory of continuous functionals, developed in \cite{nonsmooththeory1, corvellec1993deformation,degiovanni1994critical}. In this setting, we consider $X$ a metric space and $f: X \to \R$ a continuous functional.

\begin{definition} \label{definition 1}
Let $X$ be a metric space and let $f: X \to \mathbb{R}$ be a continuous function. We consider $\sigma \ge 0$ such that there exist $\delta > 0$ and a continuous map $\mathscr{H}: B_{\delta}(u) \times [0, \delta] \to X$ such that
\begin{align}
\label{condition 1}
d(\mathscr{H}(v, t), v) \le t, \\
\label{condition 2}
f(\mathscr{H}(v, t)) \le f(v) - \sigma t.
\end{align}
We define 
\[
|df|(u):=\sup\left\{
\begin{aligned}
 &\text{there exist $\delta>0$ and}\\
\sigma\ge0\ :\ \ \   &\text{$\mathscr H\in C(B_\delta(u)\times[0,\delta];X)$}\\
&\text{which satisfy \eqref{condition 1} and \eqref{condition 2}}
\end{aligned}
\right\}
\]
as the weak slope of $f$ at $u$.
\end{definition}

\begin{theorem}[{\cite[Theorem 1.1.2]{nonsmooththeory1}}]\label{theorem 1}
Let $E$ be a normed space and $X \subset E$ an open subset. Let us fix $u \in X$ and $v \in E$ with $\|v\| = 1$. For each $w \in X$ we define
$$\overline{D}_+f(w)[v] := \limsup_{t \to 0^+} \frac{f(w + tv) - f(w)}{t}.$$
Then $|df|(u) \ge -\limsup_{w \to u} \overline{D}_+f(w)[v]$.
\end{theorem}

\begin{definition}\label{def p.critical}
Let $X$ be a metric space and let $f: X \to \mathbb{R}$ be continuous. We will say that $u \in X$ is a (lower) critical point if $|df|(u) = 0$. A (lower) critical point is said to be at the level $c \in \mathbb{R}$ if it is also true that $f(u) = c$.
\end{definition}

\begin{definition}\label{ps}
Let $X$ be a metric space and let $f:X\to\R$ be a continuous functional. A sequence $\{u_n\}\subset X$ is said a $(PS)_c$-sequence if
\begin{align}
\label{ps1}
&f(u_n)\to c,\\
\label{ps2}
&|df|(u_n)\to0.
\end{align}
Furthermore, we will say that $f$ satisfies the $(PS)_c$-condition if every subsequence admits a convergent subsequence in $X$. 
If the $(PS)_c$-condition holds for every $c\in\R$ we will simply write $(PS)$-condition.
\end{definition}
\begin{theorem}[Equivariant Mountain Pass, {\cite[Theorem 1.3.3]{nonsmooththeory1}}]\label{MPequi}	\hfill\\
Let $X$ be a Banach space and let $f:X\to\R$ be a continuous even functional. Suppose that
\begin{itemize}
\item $\exists\ \rho>0,\alpha>f(0)$ and a subspace $X_2\subset X$ of finite codimension such that\\ $f\ge\alpha\ on\ \partial B_{\rho}\cap X_2$,\
\item for every subspace $X_1^{(k)}$ of dimension $k$, there exists $R=R^{(k)}>0$ such that $f\le f(0)$ in $B_R^c\cap X_1^{(k)}$.
\end{itemize}
If $f$ satisfies the $(PS)_c$-condition for every $c\ge\alpha$, then there exists  a divergent sequence of critical values, namely there exists a sequence $\{u_n\}\subset X$ such that $c_n:=f(u_n)\to+\infty$.
\end{theorem}

We investigate the regularity of $f_\lambda$ with the following:
\begin{proposition}\label{differenziabilità funzionale}
   Let \( u \in H^1(\Omega) \). Then $f_\lambda$ is continuous at \( u \) and one has that
    \[|df_\lambda|(u)\ge\sup_{\substack{v\in C^\infty(\overline\Omega)\\ \|v\|_{H^1}\le1}}\left\{\int_\Omega A(x,u)\nabla u\cdot\nabla v+\lambda\int_\Omega uv+\frac12\int_\Omega \left(D_sA(x,u)\nabla u\cdot\nabla u\right) v-\int_\Omega g(x,u)v\right\}.\]
\end{proposition}
\begin{proof}
We divide the proof in three steps.

\textbf{Step 1}. The functional $f_\lambda$ is continuous.

    Let $\{u_n\}\subset H^1(\Omega)$ such that $u_n\to u$ in $H^1(\Omega)$. Then 
    \[A(x,u_n)\to A(x,u),\quad D_sA(x,u_n)\to D_sA(x,u),\quad G(x,u_n)\to G(x,u) \]
   for a.e. $x\in\Omega$, by continuity. From \eqref{a.1}
   \begin{align*}
       |A(x,u_n)\nabla u_n\cdot\nabla u_n|\le NC|\nabla u_n|^2.
   \end{align*}
   Since $|\nabla u_n|\to|\nabla u|$ in $L^2(\Omega)$, we have that there exists $c>0,\xi\in L^1(\Omega)$ such that $|\nabla u_n|^2\le c|\xi|$. Thus, by Lebesgue's Theorem of Dominated Convergence, we obtain:
   \[\int_\Omega A(x,u_n)\nabla u_n\cdot\nabla u_n\to\int_\Omega A(x,u)\nabla u\cdot\nabla u.\]
   Now, by Mean Value Theorem, there exists $\theta\in(0,1)$  such that
   \[|G(x,u_n)|=|G(x,u_n)-G(x,0)|=|g(x,\theta u_n)u_n|\le a(x)|u_n|+\theta b|u_n|^p.\]
   By Sobolev's embedding $u_n\to u$ in $L^1(\Omega)$ and in $L^p(\Omega)$, hence there exist $\omega_1,\omega_2\in L^1(\Omega)$ and $c_1,c_2>0$ such that $|u_n|\le c_1\omega_1$ and $|u_n|^p\le c_2\omega_2$. Again, from Lebesgue's Theorem we deduce that 
   \[\int_\Omega G(x,u_n)\to \int_\Omega G(x,u).\]
 Moreover, Sobolev's embedding implies also that 
 \[\lambda\int_\Omega u_n^2\to\lambda\int_\Omega u^2.\]
   Then, $f_\lambda(u_n)\to f_\lambda(u)$.
   
   \textbf{Step 2}. For every $v\in H^1(\Omega)\cap L^\infty(\Omega)$, we have that 
   \[
   \langle f_\lambda'(u),v\rangle=\int_\Omega A(x,u)\nabla u\cdot\nabla v+\lambda\int_\Omega uv+\frac12\int_\Omega \left(D_sA(x,u)\nabla u\cdot\nabla u\right)v-\int_\Omega g(x,u)v.
   \]
   We will prove that 
   \[\lim_{t\to0}\frac{f_\lambda(u+tv)-f_\lambda(u)}{t}=\int_\Omega A(x,u)\nabla u\cdot\nabla v+\lambda\int_\Omega uv+\frac12 \int_\Omega \left(D_sA(x,u)\nabla u\cdot\nabla u\right)v-\int_\Omega g(x,u)v.\]
Let $h(x,s,\xi)$ be a $C^1$-Carathéodory function differentiable in $s,\xi$, then by Mean Value Theorem there exists $\theta\in\R$ with $|\theta|\le|t|$ such that
\[\frac{h(x,u+tv,\nabla u+t\nabla v)-h(x,u,\nabla u)}{t}=D_sh(x,u+\theta v,\nabla u+\theta\nabla v)v+D_\xi h(x,u+\theta v,\nabla u+\theta\nabla v)\cdot\nabla v.\]
We consider the two cases:
\begin{itemize}
    \item $h(x,s,\xi)=G(x,s)-\frac{\lambda}{2}s^2$.
    \[\left|\frac{h(x,u+tv)-h(x,u)}{t}\right|=|g(x,u+\theta v)v|+|\lambda|\cdot|(u+\theta v)v|\le a(x)v+b\theta|v|^p+|u+\theta v|\cdot|v|\in L^1(\Omega).\]
  
    \item $h(x,s,\xi)=\frac12 A(x,s)\xi\cdot\xi$.
   \[\left|\frac{h(x,u+tv,\nabla u+t\nabla v)-h(x,u,\nabla u)}{t}\right|\le \frac{NC}{2}|\nabla(u+\theta v)|^2v+\frac{NC}{2}|\nabla(u+\theta v)\cdot\nabla v|\in L^1(\Omega).\]
\end{itemize}

Thus,
\[\langle f_\lambda'(u),v\rangle=\int_\Omega A(x,u)\nabla u\cdot\nabla v+\lambda\int_\Omega uv+\frac12\int_\Omega \left(D_sA(x,u)\nabla u\cdot\nabla u\right)v-\int_\Omega g(x,u)v.\]

\textbf{Step 3}. We have that 
\[
|df_\lambda|(u)\ge\sup_{\substack{v\in C^\infty(\overline\Omega)\\ \|v\|_{H^1}\le1}}\langle f_\lambda'(u),v\rangle.
\]
  Let $\varphi\in C^\infty(\overline\Omega)$, according to Theorem \ref{theorem 1} and Proposition \ref{differenziabilità funzionale}, we have that
    \[\|\varphi\|_{H^1}\cdot|df_\lambda|(u)\ge-\langle f_\lambda'(u),\varphi\rangle=\langle f_\lambda'(u),-\varphi\rangle.\]
    Thus, choosing  $v=-\varphi$ with $\|\varphi\|_{H^1}\le1$, we obtain that
    \[|df_\lambda|(u)\ge\langle f_\lambda'(u),v\rangle,\quad \forall\ v\in C^\infty(\overline\Omega),\ \|v\|_{H^1}\le1.\]
Taking the supremum, the thesis follows.
\end{proof}

\begin{definition}\label{definitio CPS}
    A sequence $\{u_n\}\subset H^1(\Omega)$ is said to be a Concrete Palais-Smale  at level $c$ for $f_\lambda$, $(CPS)_c$-sequence, if
\begin{itemize}
\item $f_\lambda(u_n)\to c$ in $\R$.
\item $ -div(A(x,u_n)\nabla u_n)+\frac{1}{2}D_sA(x,u_n)\nabla u_n\cdot\nabla u_n+\lambda u_n-g(x,u_n) \in  \left(H^{1}(\Omega)\right)'$ eventually as $n\to+\infty$.  
\item $-div(A(x,u_n)\nabla u_n)+\frac{1}{2}D_sA(x,u_n)\nabla u_n\cdot\nabla u_n+\lambda u_n-g(x,u_n)\to 0$ strongly in $\left(H^{1}(\Omega)\right)'$.
\end{itemize}
We will say that $f_\lambda$ satisfies the  $(CPS)_c$-condition if every  $(CPS)_c$-sequence admits a convergent subsequence in $ H^1(\Omega)$.
\end{definition}
\begin{proposition}\label{PS e CPS}
Let $c,\lambda\in\R$.
The following facts hold:
\begin{enumerate}
\item[$(i)$] if $u$ is a lower critical point for $f_\lambda$, then $u$ is a weak solution of the problem \eqref{P},
\item[$(ii)$] each $(PS)_c$-sequence is also a $(CPS)_c$-sequence,
\item[$(iii)$] $f_\lambda$ satisfies $(CPS)_c$-condition$\implies f_\lambda$ satisfies $(PS)_c$-condition.
\end{enumerate}
\end{proposition}
\begin{proof}
 \begin{itemize}
     \item[$(i)$] Let $u\in H^1(\Omega)$ be a critical point for $f_\lambda$. Then $|df_\lambda|(u)=0$ and, from Proposition \ref{differenziabilità funzionale}, we have
     \[\sup_{\substack{v\in C^\infty(\overline\Omega)\\ \|v\|_{H^1}\le1}}\left\{\int_\Omega A(x,u)\nabla u\cdot\nabla v+\lambda\int_\Omega uv+\frac12\int_\Omega \left(D_sA(x,u)\nabla u\cdot\nabla u\right) v-\int_\Omega g(x,u)v\right\}\le0.\]
     Hence, 
     \[\int_\Omega A(x,u)\nabla u\cdot\nabla v+\lambda\int_\Omega uv+\frac12\int_\Omega \left(D_sA(x,u)\nabla u\cdot\nabla u\right)v=\int_\Omega g(x,u)v,\ \ \forall\ v\in C^\infty(\overline\Omega),\]
     which implies that $u$ is a weak solution of \eqref{P}.
\item[$(ii)$] Let $\{u_n\}$ be a $(PS)_c$-sequence for $f_\lambda$. As before,
 \[\sup_{\substack{v\in C^\infty(\overline\Omega)\\ \|v\|_{H^1}\le1}}\left\{\int_\Omega A(x,u_n)\nabla u_n\cdot\nabla v+\lambda\int_\Omega u_nv+\frac12\int_\Omega \left(D_sA(x,u_n)\nabla u_n\cdot\nabla u_n\right) v-\int_\Omega g(x,u_n)v\right\}\le o(1).\]
 Thus, $\{u_n\}$ is a $(CPS)_c$-sequence.
     \item[$(iii)$] Assume that $f_\lambda$ satisfies the $(CPS)_c$-condition and let $\{u_n\}$ be a $(PS)_c$ sequence. From $(ii)$ we have that $\{u_n\}$ is also a $(CPS)_c$-sequence and the $(CPS)_c$-condition implies that there exists $\{u_{n_k}\}\subset \{u_n\}$ such that $u_{n_k}\to u$ in $H^1(\Omega)$ as $k\to+\infty$. Thus, the $(PS)_c$-condition holds.
     
 \end{itemize}
\end{proof}
\section{Regularity results}
We prove the following integrability result:
\begin{theorem}\label{brezis browder tpye result}
    Let $u\in H^1(\Omega)$ and let $\omega\in (H^1(\Omega))'$ be such that
    \begin{equation}\label{weak formulat thm3.1}
     \int_\Omega A(x,u)\nabla u\cdot\nabla v+\frac12\int_\Omega \left(D_sA(x,u)\nabla u\cdot\nabla u\right)v=\langle\omega,v\rangle,\ \ \forall\ v\in C^\infty(\overline\Omega).   
    \end{equation}
     Let $w\in H^1(\Omega)$ such that
     \[\bigg[\left(D_sA(x,u)\nabla u\cdot\nabla u\right) w\bigg]^{-}\in L^1(\Omega).\]
     
    Then we have
\[
 \left(D_s A(x,u) \nabla u \cdot \nabla u\right)w \in L^1(\Omega)
\]
and 
\[
\int_\Omega A(x,u) \nabla u \cdot \nabla w
+ \frac12 \int_\Omega \bigl( D_s A(x,u) \nabla u \cdot \nabla u \bigr) w
= \langle \omega, w \rangle.
\]
\end{theorem}
In the Dirichlet case, see \cite[Theorem 2.2.1]{nonsmooththeory1}, this result is a consequence of the following Brezis-Browder  Theorem: 
\begin{theorem}[Brezis-Browder, '78, \cite{brezisbrowder1}]
    Let $\Omega\subseteq\R^N$ be an open set. Let $T\in H^{-1}(\Omega)\cap L^1_{\text{loc}}(\Omega)$ and $u\in H_0^1(\Omega)$ such that $T(x)u(x)\ge g(x)$ a.e. in $\Omega$ with $g\in L^1(\Omega)$. Then $T(x)u(x)\in L^1(\Omega)$ and 
    \[\langle T,u\rangle=\int_\Omega T(x)u(x)\,dx.\]
\end{theorem}
However, in our case, if we replace \( H_0^1(\Omega) \) with \( H^1(\Omega) \), the result no longer holds. Instead, we require the assumption that \( T \in L^1(\Omega) \), that is,
\[
D_s A(x,u) \nabla u \cdot \nabla u \in L^1(\Omega).
\]
This integrability follows from assumption~\eqref{a.1}, and below we provide a proof of Theorem \ref{brezis browder tpye result}

\begin{proof}[Proof of Theorem \ref{brezis browder tpye result}]
We observe that the result is obvious for $w\in H^1(\Omega)\cap L^\infty(\Omega)$, by Riesz's Theorem. Let $\{\tilde v_n\}\subset C^\infty(\overline\Omega)$ such that $v_n\to w$ in $H^1(\Omega)$ and we consider, as in \cite{brezisbrowder1}, \[v_n=w\bigg(w^2+\frac{1}{n^2}\bigg)^{-1/2}\min\bigg\{\bigg(w^2+\frac{1}{n^2}\bigg)^{1/2}-\frac1n,\bigg( \tilde v_n^2+\frac{1}{n^2}\bigg)^{1/2}-\frac 1n\bigg\}\in H^1(\Omega)\cap L^\infty(\Omega).\]
We point out that
\begin{itemize}
    \item $v_n(x)=w(x)\lambda_n(x)$ with $0\le\lambda_n\le1$ a.e. in $\Omega$,
    \item $v_n\to w$ in $H^1(\Omega)$,
    \item $|v_n|\le |w|$ a.e. in $\Omega$.
\end{itemize}
 We define 
 \[T(x):=D_sA(x,u)\nabla u\cdot\nabla u.\]
 Hence,
\[T(x)v_n(x)=\left(D_sA(x,u)\nabla u\cdot\nabla u\right)\lambda_n(x)w=\lambda_n(x)T(x)w(x),\]
and 
\[Tv_n=\lambda_nTw\ge-\lambda_n\left[\left(D_sA(x,u)\nabla u\cdot\nabla u\right)w\right]^-\ge-\left[\left(D_sA(x,u)\nabla u\cdot\nabla u\right)w\right]^-.\]
 According to Fatou's Lemma, we have that
\begin{align*}
    \int_{\Omega} \left(D_sA(x,u)\nabla u\cdot\nabla u\right)w&\le\liminf_{n\to+\infty}\int_{\Omega}\left(D_sA(x,u)\nabla u\cdot\nabla u\right)v_n\\
    &=2\liminf_{n\to+\infty}\bigg\{\int_\Omega A(x,u)\nabla u\cdot\nabla v_n-\langle\omega,v_n\rangle\bigg\}\le C,
\end{align*}
for some $C>0$.
   Hence $\left(D_sA(x,u)\nabla u\cdot\nabla u\right)w\in L^1(\Omega)$. Testing \eqref{weak formulat thm3.1} with $v_n$, we can pass to the limit according to Lebesgue's Theorem and we obtain:
     \[\int_\Omega A(x,u)\nabla u\cdot\nabla w+\frac12\int_\Omega\left( D_sA(x,u)\nabla u\cdot\nabla u\right)w=\langle\omega,w\rangle.\qedhere\]
\end{proof}
\begin{remark}\label{density remark}
   Theorem \ref{brezis browder tpye result} ensures that every $v\in H^1(\Omega)$ such that
     \begin{equation}\label{condition test function}
     \bigg[\left(D_sA(x,u)\nabla u\cdot\nabla u\right) v\bigg]^{-}\in L^1(\Omega)
     \end{equation}
     is an admissible test function for \eqref{P}. In particular, the condition \eqref{condition test function} is satisfied for $v\in H^1(\Omega)\cap L^\infty(\Omega)$ and also  for $v\equiv u$ by \eqref{a.3}.
\end{remark}

We state a regularity result:
\begin{theorem}\label{regularity}
    Let $\lambda\in\R,$ 
    and let $u\in H^1(\Omega)$ be a weak solution of \eqref{P}.
   Assume also that \eqref{g.1}  holds with $r>\frac{N}{2}$. Then $u\in L^\infty(\Omega)$.
\end{theorem}
\begin{proof} We follow
\cite[Lemma 2.2]{saldana2022least}.
    Let $\varphi\in C^\infty(\overline\Omega)$ with $\text{supp}(\varphi)\subset\{x\in\Omega\ :\ |u(x)|\ge R\}$  and let $s,L\ge0$.
  We consider $v:=u\min\{|u|^{2s},L^2\}\varphi^2\in H^1(\Omega)$. 
  By \eqref{a.3}, we can test \eqref{P} with $v$ and obtain:
\[\int_\Omega A(x,u)\nabla u\cdot\nabla v\le\int_\Omega (g(x,u)-\lambda u)v.\]
We start treating the left hand side:
\begin{align*}
        \int_\Omega A(x,u)\nabla u\cdot\nabla v&=\int_\Omega \left(A(x,u)\nabla u\cdot\nabla u\right)\left(\min\{|u|^s,L\}\varphi\right)^2\\
        &\quad+2\int_\Omega A(x,u)\nabla u\cdot\nabla(\min\{|u|^s,L\}\varphi)\left(\min\{|u|^s,L\}\varphi\right)u\\
        &=\int_\Omega A(x,u)\nabla u\cdot\nabla u\left(\min\{|u|^s,L\}\varphi\right)^2\\
        &\quad+2\int_\Omega \left(A(x,u)\nabla u\cdot\nabla\varphi\right)\left(\min\{|u|^s,L\}\right)^2u\varphi\\
        &\quad +2s\int_{\{L\ge|u|^s\}}A(x,u)\nabla u\cdot\nabla u |u|^{2s}\varphi^2.
    \end{align*}
    We set $\psi:=u\min\{|u|^s,L\}\varphi$ and we compute $|\nabla\psi|_2^2$:
    \begin{align*}
        \nu\int_\Omega |\nabla \psi|^2&\le\int_\Omega A(x,u)\nabla\psi\cdot\nabla\psi=\int_\Omega A(x,u) \nabla u\cdot\nabla u\left(\min\{|u|^s,L\}\varphi\right)^2\\
        &\quad+2\int_\Omega A(x,u)\nabla u\cdot\nabla(\min\{|u|^s,L\}\varphi) u\min\{|u|^s,L\}\varphi\\
        &\quad+\int_\Omega A(x,u)\nabla(\min\{|u|^s,L\}\varphi)\cdot \nabla(\min\{|u|^s,L\}\varphi)u^2\\
        &=\int_\Omega A(x,u)\nabla u\cdot\nabla u (\min\{|u|^s,L\}\varphi)^2 +2\int_\Omega A(x,u)\nabla u\cdot\nabla\varphi(\min\{|u|^s,L\})^2u\varphi\\
        &\quad+(2s+s^2)\int_{\{L\ge|u|^s\}}A(x,u)\nabla u\cdot\nabla u|u|^{2s}\varphi^2+\int_\Omega A(x,u)\nabla\varphi\cdot\nabla\varphi(\min\{|u|^s,L\}\varphi)^2u^2\\
        &\quad+2s\int_{\{L\ge|u|^s\}} A(x,u)\nabla u\cdot\nabla\varphi |u|^{2s}\varphi u\le C_1(s)\int_\Omega A(x,u)\nabla u\cdot\nabla v+c|u|_{2s+2}^{2s+2}\\
        &\le C_1(s)\left[\int_\Omega (g(x,u)-\lambda u)v+c|u|_{2s+2}^{2s+2}\right],
    \end{align*}
    where $C_1(s)=\max\left\{1,\frac{s+2}{2},s+1\right\}$.
    Exploiting \eqref{g.1}, we have that:
    \begin{align*}
        \int_\Omega (g(x,u)-\lambda u)v&\le\int_\Omega a(x)\varphi+b\int_\Omega |u|^{p-1}v-\lambda\int_\Omega uv\\
        &=\int_\Omega a(x)u(\min\{|u|^s,L\}\varphi)^2+b\int_\Omega |u|^p(\min\{|u|^s,L\}\varphi)^2\\
        &\quad-\lambda\int_\Omega u^2(\min\{|u|^s,L\}\varphi)^2\\
        &\le C_3+\int_\Omega (|a(x)|+|\lambda|)u^2(\min\{|u|^s,L\}\varphi)^2+\tilde b\int_{\{|u|>R+1\}} |u|^{2s+p}.
    \end{align*}
  We set $\alpha(x):=a(x)+|\lambda|$. Hence, 

 \begin{align*}
       \int_\Omega|\alpha(x)|\psi^2&\le K\int_\Omega |u|^{2s+2}\varphi^2+\int_{\{|\alpha(x)|\ge K\}}\alpha u^2\min\{|u|^{2s},L^2\}\varphi^2 \\
        &\le C_2K|u|_{2s+2}^{2s+2}+\left(\int_{\{|\alpha(x)|\ge K\}}|\alpha(x)|^{\frac N2}\right)^{\frac 2N}\cdot\left(\int_\Omega |u\min\{|u|^s,L\}\varphi|^{\frac{2N}{N-2}}\right)^{\frac{N-2}{N}}\\
        &\le C_2K|u|_{2s+2}^{2s+2}+\delta(K)\left(\int_\Omega|\psi|^{\frac{2N}{N-2}}\right)^{\frac{N-2}{N}},\quad \text{with $\delta(K)\to0$ a $K\to+\infty$}.
    \end{align*}
   Now, as in  \cite[page 30, proof of Lemma 2.2]{saldana2022least} and since $p>2$, we get:
\begin{align*}
    \left(\int_\Omega|\psi|^{\frac{2N}{N-2}}\right)^{\frac{N-2}{N}}&\le k_1\left(\int_\Omega\left|\psi-\int_\Omega\psi\right|^{\frac{2N}{N-2}}\right)^{\frac{N-2}{N}}+k_2\left(\int_\Omega|\psi|\right)^2\\
    &\le k_3\int_\Omega|\nabla \psi|^2+k_4\left(\int_\Omega|u|^{s+1}\right)^2\\
    &\le C_1(s)C_3+C_1(s)C_2K|u|_{2s+p}^{2s+p}+\delta(K)|\psi|_{2^*}^2+\tilde b|u|_{2s+p}^{2s+p}.
\end{align*}
Thus, 
\[|\psi|_{2^*}^2\le \kappa(s)|u|_{2s+p}^{2s+p},\]

    where $\kappa$ is independent on $L$. Then, passing to the limit as $L\to+\infty$ we conclude that $u\in L^{\frac{2(s+1)N}{N-2}}(\Omega)$ for every $s\ge0$ and we obtain $u\in L^t(\Omega)$ for every $t\ne\infty$.
 
       In particular, there exists $m>\frac{N}{2}$ such that
    \[
    h(x,u):=g(x,u)-\lambda u+\xi(u)\in L^m(\Omega).
    \]
    We define 
    \[
    \xi(u)=
    \begin{cases}
        u+R&\text{if $ u(x)<-R$}\\
        0&\text{if $|u(x)|\le R$}\\
        u-R&\text{if $u(x)>R$}.
    \end{cases}
    \]
    Thus,
\begin{align*}
    \int_\Omega A(x,u)\nabla u\cdot\nabla\xi(u)+\int_\Omega |\xi(u)|^2+\frac12\int_\Omega \xi(u)D_sA(x,u)\nabla u\cdot\nabla u=\int_\Omega h(x,u)\xi(u),
\end{align*}
with
\[\int_\Omega A(x,u)\nabla u\cdot\nabla\xi(u)=\int_{\{|u|\ge R\}}A(x,u)\nabla u\cdot\nabla u,\qquad \int_\Omega \xi(u)D_sA(x,u)\nabla u\cdot\nabla u\ge0.\]
Therefore, 
    \begin{align*}
       \nu\left(\int_\Omega|\xi(u)|^{\frac{2N}{N-2}}\right)^{\frac{N-2}{N}}&\le \nu C\int_{\{|u|>R\}}(|\nabla u|^2+|\xi(u)|^2)= \nu C\int_\Omega(|\nabla \xi(u)|^2+|\xi(u)|^2)\\ 
       &\le C\left[\int_\Omega A(x,u)\nabla u\cdot\nabla \xi(u)+\int_\Omega |\xi(u)|^2+\int_{\Omega}\xi(u)D_sA(x,u)\nabla u\cdot\nabla u\right]\\
       &\le C\int_\Omega h(x,u)\xi(u)\le C\left(\int_{\{|u|>R\}}|h(x,u)|^{\frac{2N}{N+2}}\right)^{\frac{N+2}{2N}}\cdot\left(\int_{\Omega}|\xi(u)|^{\frac{2N}{N-2}}\right)^{\frac{N-2}{2N}}.
    \end{align*}
Hence,
\begin{align*}
\nu\left(\int_\Omega|\xi(u)|^{\frac{2N}{N-2}}\right)^{\frac{N-2}{2N}}&\le C\left(\int_{\{|u|>R\}}|h(x,u)|^{\frac{2N}{N+2}}\right)^{\frac{N+2}{2N}}\\
&\le C \left(\int_{\{|u|>R\}}|h(x,u)|^{m}\right)^{\frac{1}{m}}|\{|u|>R\}|^{\frac{N+2}{2N}-\frac1m}.
\end{align*}
On the other hand, for every $M>R$ we have that
\begin{align*}
    \left(\int_\Omega|\xi(u)|^{\frac{2N}{N-2}}\right)^{\frac{N-2}{2N}}\ge \left(\int_{\{|u|>M\}}|\xi(u)|^{\frac{2N}{N-2}}\right)^{\frac{N-2}{2N}}\ge (M-R)|\{|u|>M\}|^{\frac{N-2}{2N}}.
\end{align*}
Therefore,
\[|\{|u|>M\}|\le\frac{C_\nu}{M-R}|\{|u|>R\}|^{\gamma},\quad\gamma=\frac{2N}{N-2}\left(\frac{N+2}{2N}-\frac{1}{m}\right)>1.\]
According to \cite[Lemma 4.1]{stampacchia1965probleme}, there exists $K>0$ such that $|\{|u|>M\}|=0$ for every $M\ge K$ and this implies that $u\in L^\infty(\Omega)$.

\end{proof}

We stress the fact that the hypothesis \eqref{a.3} is essential to obtain this regularity. We present the example in \cite[Introduction, page 5]{nonsmooththeory1}.
\begin{example}
Let $N\ge3$, let $\Omega=B_1\subset\R^N$ be a ball centered at $x=0$ with radius $r=1$,  and let $A(x,s)=a(x,s)\text{Id}$ with
\[a(x,s)=\left(1+\frac{1}{|x|^{\alpha}e^s+1}\right),\quad \alpha=12(N-2),\]
where Id is the identity matrix. Then $u(x)=-\alpha\log|x|\in H^1(B_1)$ is a weak solution of
\[
\begin{cases}
-\text{div}(A(x,u)\nabla u)+\frac12 D_sA(x,u)\nabla u\cdot\nabla u=0 &\text{in $B_1$},\\
A(x,u)\nabla u\cdot\nabla\eta=-\frac32\alpha&\text{on $\partial B_1$},
\end{cases}
\]
but 
\[D_sa(x,s)=-\frac{|x|^\alpha e^s}{(|x|^\alpha e^s+1)^2}\le0,\]
and $u\notin L^\infty(B_1)$. Indeed, $A(x,u)=\frac32\text{Id}$, $\nabla u=-\alpha\frac{x}{|x|^2}$, $\eta=\frac{x}{|x|}$ and
\begin{align*}
-\text{div}(A(x,u)\nabla u)+\frac12 D_sA(x,u)\nabla u\cdot\nabla u&=-\frac32\Delta u-\frac{1}{8}|\nabla u|^2=0
\end{align*}
in the weak sense, but $u\in H^1(B_1)$ has a singularity  at $x=0$. Thus, $u$ is a singular weak solution of a quasilinear elliptic problem where data, boundary conditions and domain are $C^\infty$. Notice also that
\begin{align*}
\int_{B_1}|\log|x||^q\,dx&=\left(\int_{\partial B_1}\,d\sigma\right)\left(\int_0^1|\log(r)|^qr^{N-1}\,dr\right)\\
&=\omega_{N-1}\int_0^{+\infty}\rho^qe^{-(N-1)\rho}e^{-\rho}\,d\rho=\omega_{N-1}\int_0^{+\infty}\rho^qe^{-N\rho}\,d\rho\\
&=\frac{\omega_{N-1}}{N^{q+1}}\Gamma(q+1),
\end{align*}
where $\omega_{N-1}$ is the area of the $(N-1)$-sphere and $\Gamma$ is the Gamma function.
Therefore, $u\in L^q(B_1)$ for every $q\ne\infty$, but $u\notin L^\infty(B_1)$.
\end{example}

\section{The Palais-Smale condition}
\begin{theorem}\label{PS condition}
Let $\{u_n\}\subset H^1(\Omega)$ be a bounded sequence  such that the sequence $\{\omega_n\}$, where
\[
\omega_n:=-div(A(x,u_n)\nabla u_n)+\frac12D_sA(x,u_n)\nabla u_n\cdot\nabla u_n,
\]
is strongly convergent in $(H^1(\Omega))'$.  Then $\{u_n\}$ admits a strongly convergent subsequence in $H^1$ to some $u\in H^1(\Omega)$.
\end{theorem}
 \begin{proof}
The proof is similar to \cite[Lemma 3.4]{caninoserdica} with some minor modifications, but we provide it for completeness.
Let $\omega$ be the limit in $(H^1(\Omega))'$ of $\{\omega_n\}$. Since $\{u_n\}$ is bounded in $H^1(\Omega)$, we can assume (up to a subsequence) that $u_n\wto u$ in $H^1$ and
\begin{align*}
&u_n\to u\ \ \text{in $L^2(\Omega)$},&u_n(x)\to u(x)\ \ \text{a.e. $x\in\Omega$.}\\
\end{align*}
We point out that we can also assume that $\nabla u_n(x)\to\nabla u(x)$ for a.e. $x\in\Omega$ by \cite[Remark 2.1]{boccardo92Gradientconvergence}. 

\textbf{Step 1}: \hspace{1cm}$\displaystyle{\langle \omega,u\rangle=\int_{\Omega}A(x,u)\nabla u\cdot\nabla u+\frac12\int_{\Omega}\left(D_sA(x,u)\nabla u\cdot\nabla u\right) u}$.

We consider the test function \[v_n=\varphi e^{-M(u_n+R)^+},\]
where $\varphi\in H^1(\Omega)\cap L^\infty(\Omega)$ with $\varphi\ge0$ and
let $M>0$ be such that
\begin{equation}\label{cond 3.1}
\frac12\left|D_sA(x,s)\xi\cdot\xi\right|\le 
MA(x,s)\xi\cdot\xi,
\end{equation}
for every $s\in\R$ and for every $\xi\in\R^N$ (existence of such constant follows by \eqref{a.1}-\eqref{a.2}).

Hence, $v_n\in H^1(\Omega)\cap L^\infty(\Omega)$ is an admissible test function (according to Remark \ref{density remark}), and
$$\int_{\Omega}A(x,u_n)\nabla u_n\cdot\nabla\left(\varphi e^{-M(u_n+R)^+}\right)+\frac12\int_{\Omega}\left(D_s A(x,u_n)\nabla u_n\cdot \nabla u_n\right)\varphi e^{-M(u_n+R)^+}=\langle \omega_n,v_n\rangle.$$
By the product derivation rule, the quantity to the left hand side becomes:
\begin{align*}
\int_{\Omega}A(x,u_n)\nabla u_n\cdot\nabla(\varphi) e^{-M(u_n+R)^+}&+\int_{\Omega}\biggl[\frac12D_sA(x,u_n)\nabla u_n\cdot\nabla u_n\\
&-M A(x,u_n)\nabla u_n\cdot\nabla (u_n+R)^+\biggr]v_n.
\end{align*}
From \eqref{a.3} and the inequality
\begin{equation}
\label{revised 3.1}
\frac12D_sA(x,s)\xi\cdot\xi\le\frac12|D_sA(x,s)\xi\cdot\xi|\le M A(x,s)\xi\cdot\xi,
\end{equation}
we obtain that
\begin{equation}
\label{segno integrale 1}
\frac12D_sA(x,u_n)\nabla u_n\cdot\nabla u_n-M A(x,u_n)\nabla u_n\cdot\nabla (u_n+R)^+\le 0.
\end{equation}
Indeed, if $u_n+R\le0$ we have that $\nabla (u_n+R)^+=0$ and  $D_sA(x,u_n)\nabla u_n\cdot\nabla u_n\le0$ by \eqref{a.3}, hence \eqref{segno integrale 1} holds.
 Otherwise $\nabla (u_n+R)^+=\nabla u_n$, and \eqref{segno integrale 1} follows by \eqref{revised 3.1}. 
 
 According to Fatou's Lemma, we get:
\begin{align*}
\int_{\Omega}\limsup_{n\to+\infty}\biggl\{\left(A(x,u_n)\nabla u_n\cdot\nabla\varphi\right)e^{-M(u_n+R)^+}&+\biggl[\frac12D_sA(x,u_n)\nabla u_n\cdot\nabla u_n\\
&-M A(x,u_n)\nabla u_n\cdot\nabla (u_n+R)^+\biggr]v_n\biggl\}\\
&\quad\ge\langle \omega,v_n\rangle.
\end{align*}
Now, let
$$\varphi_k=\varphi H\left(\frac{u}{k}\right)e^{M(u+R)^+},$$
where $\varphi\in C^{\infty}(\overline\Omega), \varphi\ge0$, $H\in C^1(\R)$ with $0\le H\le 1$, $H(s)=1$ if $s\in[-\frac12,\frac12]$ and $H(s)=0$ in $(-1,1)^c$. Then $\varphi_k\in H^1(\Omega)\cap L^{\infty}(\Omega)$ and $\varphi_k\ge0$. Therefore:
\begin{align*}
\int_{\Omega}A(x,u)\nabla u\cdot\nabla\left(\varphi H\left(\frac{u}{k}\right)\right)+\frac12\int_{\Omega}\left(D_sA(x,u)\nabla u\cdot\nabla u\right)\varphi H\left(\frac{u}{k}\right)\ge\left\langle \omega,\varphi H\left(\frac{u}{k}\right)\right\rangle.
\end{align*}
Taking the limit as $k\to+\infty$:
$$\int_{\Omega}A(x,u)\nabla u\cdot\nabla\varphi+\frac12\int_{\Omega}\left(D_sA(x,u)\nabla u\cdot\nabla u\right)\varphi\ge\langle \omega,\varphi\rangle.$$
Considering $\tilde v_n=\varphi e^{-M(u_n-R)^-}$, we obtain the opposite inequality:
$$\frac12D_sA(x,u_n)\nabla u_n\cdot \nabla u_n-M A(x,u_n)\nabla u_n\cdot\nabla (u_n-R)^-\ge0.$$
Then 
\begin{equation}\label{eq omega}
    \int_{\Omega}A(x,u)\nabla u\cdot\nabla\varphi+\frac12\int_{\Omega}\left(D_sA(x,u)\nabla u\cdot\nabla u\right)\varphi=\langle \omega,\varphi\rangle,
\end{equation}
for each $\varphi\in C^{\infty}(\overline\Omega)$ with  $\varphi\ge0$.

 By linearity, the identity also holds for every test function $\varphi\in C^\infty(\overline\Omega)$. Indeed, the density of \( C^\infty(\overline\Omega) \) in \( H^1(\Omega) \cap L^\infty(\Omega) \) implies that \eqref{eq omega} holds for every \( \varphi \in H^1(\Omega) \cap L^\infty(\Omega) \) with \( \varphi \geq 0 \) (see also Remark \ref{density remark}). Taking \( -\varphi \), we obtain \eqref{eq omega} also for non-positive test functions. Now, let \( \varphi \in C^\infty(\overline\Omega) \) (possibly sign-changing). Then, we can write \( \varphi = \varphi^+ - \varphi^- \), where \( \varphi^+ \geq 0 \) and \( -\varphi^- \leq 0 \). By linearity, we infer that \eqref{eq omega} also holds for any \( \varphi \in C^\infty(\overline\Omega) \).
 Finally, according to  Remark \ref{density remark}, we get also \eqref{eq omega} for $\varphi=u$.\\

\textbf{Step 2}: \hspace{1cm}$\displaystyle{\limsup_{n\to+\infty}\int_{\Omega}e^{\zeta(u_n)}A(x,u_n)\nabla u_n\cdot\nabla u_n\le\int_{\Omega}e^{\zeta(u)}A(x,u)\nabla u\cdot\nabla u}$, where
\[
\zeta(s)=
\begin{cases}
M s&0<s<R\\
M R&s\ge R\\
-M s&-R<s<0\\
M R&s\le-R.
\end{cases}
\]
We take $\tilde v_n=u_ne^{\zeta(u_n)}$ and the quantity
$$\left[\left(D_sA(x,u_n)\nabla u_n\cdot\nabla u_n\right)u_ne^{\zeta(u_n)}\right]^{-}.$$
If $|u_n|\ge R$ this negative part is zero, while $|u_n|<R$, $\zeta(u_n)=|u_n|\le M R$. Therefore,
$$ e^{\zeta(u_n)}\le e^{M R}\implies \left(D_sA(x,u_n)\nabla u_n\cdot\nabla u_n\right) u_ne^{\zeta(u_n)}<\tilde cD_sA(x,u_n)\nabla u_n\cdot\nabla u_n\in L^1(\Omega),$$
and $\tilde v_n$ is an admissible test function from Theorem \ref{brezis browder tpye result}. Therefore:
$$\int_{\Omega}A(x,u_n)\nabla u_n\cdot\nabla\left(u_ne^{\zeta(u_n)}\right)+\frac12\int_{\Omega}\left(D_sA(x,u_n)\nabla u_n\cdot \nabla u_n\right)u_ne^{\zeta(u_n)}=\langle\omega_n,\tilde v_n\rangle.$$
 Since $$\left(A(x,u_n)\nabla u_n\cdot\nabla u_n\zeta'(u_n)+\frac12D_sA(x,u)\nabla u_n\cdot\nabla u_n\right)u_ne^{\zeta(u_n)}\ge0,$$ we can apply Fatou's Lemma and we get that
\begin{align*}
\limsup_{n\to+\infty}\int_{\Omega}e^{\zeta(u_n)}A(x,u_n)\nabla u_n\cdot\nabla u_n\le\int_{\Omega}e^{\zeta(u)}A(x,u)\nabla u\cdot\nabla u.
\end{align*}
\textbf{Step 3}: we have that $u_n\to u$  strongly in $H^1(\Omega)$.
\begin{align*}
\int_{\Omega}e^{\zeta(u_n)}A(x,u_n)\nabla(u_n-u)\cdot\nabla(u_n-u)&=\int_{\Omega}e^{\zeta(u_n)}A(x,u_n)\nabla u_n\cdot\nabla u_n\\
&\quad-2\int_{\Omega}e^{\zeta(u_n)}A(x,u_n)\nabla u\cdot\nabla u_n\\
&\quad+\int_{\Omega}e^{\zeta(u_n)}A(x,u_n)\nabla u\cdot\nabla u.
\end{align*}
On the other hand:
$$\lim_{n\to+\infty}\sum_{i=1}^Ne^{\zeta(u_n)}a_{ij}(x,u_n)\frac{\partial u}{\partial x_i}=\sum_{i=1}^Ne^{\zeta(u)}a_{ij}(x,u)\frac{\partial u}{\partial x_i}$$
strongly in $L^2(\Omega)$ for every $j=1,\dots,N$. Then, from Step 2:
\begin{align*}
\limsup_{n\to+\infty}\int_{\Omega}e^{\zeta(u_n)}A(x,u_n)\nabla(u_n-u)\cdot\nabla(u_n-u)&=\limsup_{n\to+\infty}\int_{\Omega}e^{\zeta(u_n)}A(x,u_n)\nabla u_n\cdot\nabla u_n\\
&-\int_{\Omega}e^{\zeta(u)}A(x,u)\nabla u\cdot\nabla u\le0.
\end{align*}
Since $e^{\zeta(u_n)}\ge1$ and from \eqref{a.2}, we obtain that:
\[\nu\limsup_{n\to+\infty}\int_\Omega|\nabla (u_n-u)|^2\le\limsup_{n\to+\infty}\int_{\Omega}e^{\zeta(u_n)}A(x,u_n)\nabla(u_n-u)\cdot\nabla(u_n-u)\le0.\]
Hence, the strong convergence in $L^2(\Omega)$ implies also that $u_n\to u$ strongly in $H^1(\Omega)$.
\end{proof}
\begin{corollary}\label{CPS condition}
    For every $c\in\R$ and $\lambda\in\R$, the following facts are equivalent:
    \begin{itemize}
        \item[$(a)$] $f_\lambda$ satisfies the $(CPS)_c$-condition,
        \item[$(b)$] every $(CPS)_c$-sequence for $f_\lambda$ is bounded in $H^1(\Omega)$.
    \end{itemize}
\end{corollary}
\begin{proof}
    We prove the difficult part $(b)\implies (a)$. Let $\{u_n\}$ be a bounded $(CPS)_c$-sequence for $f_\lambda$. Then, up to a subsequence, $u_n\wto u$ in $H^1(\Omega)$ and $u_n\to u$ in $L^p(\Omega)$ for every $p\in[1,2^*)$. Hence, 
    \begin{align*}
        &\lambda\int_\Omega u_nv\to\lambda\int_\Omega uv,&\int_\Omega g(x,u_n)v\to\int_\Omega g(x,u)v,
    \end{align*}
    for every $v\in H^1(\Omega)$. Since $\{u_n\}$ is a $(CPS)$-sequence, we have
    \[\int_\Omega A(x,u_n)\nabla u_n\cdot\nabla v+\frac12\int_\Omega\left( D_sA(x,u_n)\nabla u_n\cdot\nabla u_n\right)v=\langle\omega_n,v\rangle,\]
    with $\omega_n=g(x,u_n)-\lambda u_n+\omega_n'$ and $\omega_n'\to0$ in $(H^1(\Omega))'$. Therefore, up to a subsequence, $u_n\to u$ in $H^1(\Omega)$ by Theorem \ref{PS condition}. 
\end{proof}
\begin{proposition}\label{lemma bound}
    Let $c\in\R$ and let $\{u_n\}\subset H^1(\Omega)$ be a $(CPS)_c$-sequence for $f_\lambda$. Then for every $\rho>0, \varepsilon>0$ there exists $K(\rho,\varepsilon)>0$ such that for every $n\in\N$ 
\begin{align*}
\int_{\Omega_1^\rho}A(x,u_{n})\nabla u_n\cdot\nabla u_n\le\varepsilon\int_{\Omega_2^\rho}A(x,u_{n})\nabla u_n\cdot\nabla u_n+K(\rho,\varepsilon),
\end{align*}
\begin{align*}
    &\Omega_1^{\rho}=\{x\in\Omega\ :\ |u_{n}|\le \rho\},&\Omega_2^{\rho}=\{x\in\Omega\ :\ |u_{n}|>\rho\}.
\end{align*} 
\end{proposition}
\begin{proof}
The proof is similar to that of \cite[Theorem 4.4]{caninoserdica}, but we include it here for the sake of completeness.

Let 
\[
w_n:=-div(A(x,u_n)\nabla u_n)+\frac12D_sA(x,u_n)\nabla u_n\cdot\nabla u_n-(g(x,u_n)-\lambda u_n),
\]
 $\sigma>0$ and  
\[
\theta_1(s)=
\begin{cases}
s&|s|<\sigma\\
-s+2\sigma&\sigma\le s<2\sigma\\
-s-2\sigma&-2\sigma<s\le-\sigma\\
0&|s|\ge2\sigma.
\end{cases}
\]
Since $\langle w_n,\theta_1(u_n)\rangle\le\|w_n\|_{(H^1)'}\cdot\|\theta_1(u_n)\|_{H^1}$, we obtain
\begin{align*}
\int_{\Omega}A(x,u_n)\nabla u_n\cdot\nabla(\theta_1(u_n))&+\frac12\int_{\Omega}\left(D_sA(x,u_n)\nabla u_n\cdot\nabla u_n\right)\theta_1(u_n)\le\int_{\Omega}(g(x,u_n)-\lambda u_n)\theta_1(u_n)\\
&+\|w_n\|_{(H^1)'}\cdot\|\theta_1(u_n)\|_{H^1}.
\end{align*}
From \eqref{g.1} and weighted Young's inequality, we can write:
\begin{align*}
\int_{\{|u_n|\le\sigma\}}A(x,u_n)\nabla u_n\cdot\nabla u_n&-\int_{\{\sigma<|u_n|<2\sigma\}}A(x,u_n)\nabla u_n\cdot\nabla u_n\\
&+\frac12\int_{\{|u_n|\le\sigma\}}\left(D_sA(x,u_n)\nabla u_n\cdot\nabla u_n\right)\theta_1(u_n)\\
&+\frac12\int_{\{\sigma<|u_n|\le2\sigma\}}\left(D_sA(x,u_n)\nabla u_n\cdot\nabla u_n\right)\theta_1(u_n)\le\int_{\Omega}\left[a(x)+b|2\sigma|^{p-1}\right]\sigma\\
&+\kappa_1(|\lambda|)\sigma^2+\frac{1}{\nu}\|w_n\|_{(H^1)'}^2+\frac{\nu}{4}\|\theta_1(u_n)\|_{H^1}^2.
\end{align*}
Let $M>0$ be such that \eqref{cond 3.1} holds
and let $K_0>0$ be such that $\|w_n\|_{(H^1)'}\le K_0$. According to  \eqref{a.2}, we have:
\begin{align*}
\|\theta_1(u_n)\|_{H^1}^2&=\int_\Omega|\nabla\theta_1(u_n)|^2+\int_\Omega|\theta_1(u_n)|^2\\
&\le\frac{1}{\nu}\left[\int_{\{|u_n|\le\sigma\}}A(x,u_n)\nabla u_n\cdot\nabla u_n+\int_{\{\sigma<|u_n|\le2\sigma\}}A(x,u_n)\nabla u_n\cdot\nabla u_n\right]+\kappa_2\sigma^2.
\end{align*}
Then from \eqref{cond 3.1}:
\begin{align*}
\left(1-\sigma M-\frac14\right)\int_{\{|u_n|\le\sigma\}}A(x,u_n)\nabla u_n\cdot\nabla u_n&\le\left(1+\sigma M+\frac14\right)\int_{\{\sigma<|u_n|\le2\sigma\}}A(x,u_n)\nabla u_n\cdot\nabla u_n\\
&\quad+\int_{\Omega}\left[a(x)+b|2\sigma|^{p-1}\right]\sigma+\frac{K_0^2}{\nu}+\frac
{\nu}{4}\kappa_2\sigma^2.
\end{align*}
Choosing $\sigma=\frac{1}{2M}$:
\begin{align}\label{stima integrali eq1 cap3}
\int_{\{|u_n|\le\sigma\}}A(x,u_n)\nabla u_n\cdot\nabla u_n\le K_1\int_{\{\sigma<|u_n|\le2\sigma\}}A(x,u_n)\nabla u_n+K_2.
\end{align}
Similarly, if we consider
\[
\theta_2(s)=
\begin{cases}
0&|s|\le\sigma\\
s-\sigma&\sigma<s<2\sigma\\
s+\sigma&-2\sigma<s<-\sigma\\
-s+3\sigma&2\sigma\le s<3\sigma\\
-s-3\sigma&-3\sigma<s\le-2\sigma\\
0&|s|\ge3\sigma,
\end{cases}
\]
we get
\begin{align}\label{stima integrali eq2 cap3}
\int_{\{\sigma<|u_n|\le2\sigma\}}A(x,u_n)\nabla u_n\cdot\nabla u_n\le K_1'\int_{\{2\sigma<|u_n|\le3\sigma\}}A(x,u_n)\nabla u_n\cdot\nabla u_n+K_2',
\end{align}
which can be written as follows:
\begin{align*}
\int_{\{|u_n|\le2\sigma\}}A(x,u_n)\nabla u_n\cdot\nabla u_n&-\int_{\{|u_n|\le\sigma\}}A(x,u_n)\nabla u_n\cdot\nabla u_n\\
&\quad\le K_1'\int_{\{2\sigma<|u_n|\le3\sigma\}}A(x,u_n)\nabla u_n\cdot\nabla u_n.
\end{align*}
On the other hand, by \eqref{stima integrali eq1 cap3}-\eqref{stima integrali eq2 cap3}:
\begin{align*}
-\int_{\{|u_n|\le\sigma\}}A(x,u_n)\nabla u_n\cdot\nabla u_n&\ge-K_1\int_{\{\sigma<|u_n|\le2\sigma\}}A(x,u_n)\nabla u_n\cdot\nabla u_n-K_2\\
&\ge-K_1K_1'\int_{\{2\sigma<|u_n|\le3\sigma\}}A(x,u_n)\nabla u_n\cdot\nabla u_n-K_1K_2'-K_2.
\end{align*}
Therefore
\begin{align*}
\int_{\{|u_n|\le2\sigma\}}A(x,u_n)\nabla u_n\cdot\nabla u_n\le K_1''\int_{\{2\sigma<|u_n|\le3\sigma\}}A(x,u_n)\nabla u_n\cdot\nabla u_n+K_2''.
\end{align*}
Iterating the procedure, for each $k\ge1$ we get
\begin{align}
\label{iterative 4.3}
\int_{\{|u_n|\le k\sigma\}}A(x,u_n)\nabla u_n\cdot\nabla u_n\le K_1^{(k)}\int_{\{k\sigma<|u_n|\le(k+1)\sigma\}}A(x,u_n)\nabla u_n\cdot\nabla u_n+K_2^{(k)}.
\end{align}
Let, now, $k\ge1$ be such that $k\sigma\ge R$ and $\delta\in(0,1)$. We consider
\[
\theta_{\delta}(s)=
\begin{cases}
0&|s|\le k\sigma\\
s-k\sigma&k\sigma<s<(k+1)\sigma\\
s+k\sigma&-(k+1)\sigma<s<-k\sigma\\
-\delta s+\sigma+\delta(k+1)\sigma&(k+1)\sigma\le s<(k+1)\sigma+\frac{\sigma}{\delta}\\
-\delta s-\sigma-\delta(k+1)\sigma&-(k+1)\sigma-\frac{\sigma}{\delta}<s\le-(k+1)\sigma\\
0&|s|\ge(k+1)\sigma+\frac{\sigma}{\delta}.
\end{cases}
\]
As in the first part of the proof:
\begin{align*}
\int_{\Omega}A(x,u_n)\nabla u_n\cdot\nabla(\theta_{\delta}(u_n))&+\frac12\int_{\Omega}\left(D_sA(x,u_n)\nabla u_n\cdot\nabla u_n\right)\theta_{\delta}(u_n)\\
&\le\int_{\Omega}(g(x,u_n)-\lambda u_n)\theta_{\delta}(u_n)+\|w_n\|_{(H^1)'}\cdot\|\theta_{\delta}(u_n)\|_{H^1}\\
&\le\int_{\Omega}(g(x,u_n)-\lambda u_n)\theta_{\delta}(u_n)+\frac{1}{4\delta}\|w_n\|_{(H^1)'}^2+\delta\|\theta_{\delta}(u_n)\|_{H^1}^2.
\end{align*}
We note that
\begin{align*}
\int_{\Omega}\left(D_sA(x,u_n)\nabla u_n\cdot\nabla u_n\right)\theta_{\delta}(u_n)\ge0.
\end{align*}
According to \eqref{a.2}:
\begin{align*}
\|\theta_{\delta}(u_n)\|_{H^1}^2&=\int_\Omega|\nabla\theta_{\delta}(u_n)|^2+\int_\Omega|\theta_{\delta}(u_n)|^2\\
&\le\frac{1}{\nu}\int_{\{k\sigma<|u_n|\le(k+1)\sigma\}}A(x,u_n)\nabla u_n\cdot\nabla u_n+\frac{1}{\nu}\int_{\{|u_n|>(k+1)\sigma\}}A(x,u_n)\nabla u_n\cdot\nabla u_n\\
&\quad+\kappa_4(\delta,\sigma,k).
\end{align*}
Thus
\begin{align*}
\left(1-\frac{\delta}{\nu}\right)\int_{\{k\sigma<|u_n|\le(k+1)\sigma\}}A(x, u_n)\nabla u_n\cdot\nabla u_n&\le\left(\delta+\frac{\delta}{\nu}\right)\int_{\{|u_n|>(k+1)\sigma\}}A(x,u_n)\nabla u_n\cdot\nabla u_n\\
&+\int_{\Omega}\left[a(x)+b\left((k+1)\sigma+\frac{\sigma}{\delta}\right)^{p-1}\right]\sigma\\
&+\frac{K_0^2}{4\delta}+\kappa_4,
\end{align*}
and
\begin{align*}
\int_{\{k\sigma<|u_n|\le(k+1)\sigma\}}A(x,u_n)\nabla u_n\cdot\nabla u_n\le\frac{\nu\delta+\delta}{\nu-\delta}\int_{\{|u_n|>(k+1)\sigma\}}A(x,u_n)\nabla u_n\cdot\nabla u_n+K_3(k,\delta).
\end{align*}
Finally, combining with the \eqref{iterative 4.3}:
\begin{align*}
\int_{\Omega_1^{\rho}}A(x,u_n)\nabla u_n\cdot\nabla u_n&\le\int_{\{|u_n|\le k\sigma\}}A(x,u_n)\nabla u_n\cdot\nabla u_n\\
&\le K_1^{(k)}\int_{\{k\sigma<|u_n|\le(k+1)\sigma\}}A(x,u_n)\nabla u_n\cdot\nabla u_n+K_2^{(k)}\\
&\le K_1^{(k)}\frac{\nu\delta+\delta}{\nu-\delta}\int_{\{|u_n|>(k+1)\sigma\}}A(x,u_n)\nabla u_n\cdot\nabla u_n+K_1^{(k)}K_3(k,\delta)+K_2^{(k)}\\
&\le K_1^{(k)}\frac{\nu\delta+\delta}{\nu-\delta}\int_{\Omega_2^{\rho}}A(x,u_n)\nabla u_n\cdot\nabla u_n+K_1^{(k)}K_3(k,\delta)+K_2^{(k)}.
\end{align*}
If we take $\delta$ such that
$$K_1^{(k)}\frac{\nu\delta+\delta}{\nu-\delta}\le\varepsilon,$$
we obtain the thesis with $K(\rho,\varepsilon):=K_1^{(k)}K_3(k,\delta_\varepsilon)+K_2^{(k)}$.
  
\end{proof}
\begin{theorem}\label{CPS bounded}
    Let $c\in\R$ and $\lambda\in\R$. Every $(CPS)_c$-sequence for $f_\lambda$ is bounded in $H^1(\Omega)$.
\end{theorem}
   \begin{proof}
       We define 
       \[\omega_n=-div(A(x,u_n)\nabla u_n)+\frac12 D_sA(x,u_n)\nabla u_n\cdot\nabla u_n-(g(x,u_n)-\lambda u_n).\]
       Since $\{u_n\}$ is $(CPS)_c$-sequence, we have that $\omega_n\to 0$ in $\left(H^1(\Omega)\right)'$. Then 
       \begin{align*}
           -\|\omega_n\|\cdot\|u_n\|\le\langle\omega_n,u_n\rangle&=\int_\Omega A(x,u_n)\nabla u_n\cdot\nabla u_n+\frac12\int_\Omega \left(D_sA(x,u_n)\nabla u_n\cdot\nabla u_n\right)u_n\\
           &\quad-\int_\Omega (g(x,u_n)-\lambda u_n)u_n.
       \end{align*}
       Let $\Omega_1^R,\Omega_2^R$ be open sets as in Proposition \ref{lemma bound}. According to hypothesis \eqref{a.4}, we have
       \begin{align*}
           \int_\Omega \left(D_sA(x,u_n)\nabla u_n\cdot\nabla u_n\right)u_n&=\int_{\Omega_1^R}\left(D_sA(x,u_n)\nabla u_n\cdot\nabla u_n\right)u_n+\int_{\Omega_2^R}\left(D_sA(x,u_n)\nabla u_n\cdot\nabla u_n\right)u_n\\
           &\le NCR\int_\Omega |\nabla u_n|^2+\gamma\int_{\Omega_2^R}A(x,u_n)\nabla u_n\cdot\nabla u_n.
       \end{align*}
       Let $\gamma'\in(\gamma,q-2)$ and $\varepsilon>0$ such that $\frac{NCR\varepsilon}{\nu}\le\gamma'-\gamma$. From Proposition \ref{lemma bound}, we obtain:
       \begin{align*}
           \int_\Omega \left(D_sA(x,u_n)\nabla u_n\cdot\nabla u_n\right)u_n&\le\frac{NCR}{\nu}\int_{\Omega_1^R}|\nabla u_n|^2+\gamma\int_{\Omega_2^R}A(x,u_n)\nabla u_n\cdot\nabla u_n\\
           &\le\left(\frac{NCR\varepsilon}{\nu}+\gamma\right)\int_{\Omega_2^R}A(x,u_n)\nabla u_n\cdot\nabla u_n+K(R,\varepsilon)\\
           &\le\gamma'\int_\Omega A(x,u_n)\nabla u_n\cdot\nabla u_n+K(R,\varepsilon).
       \end{align*}
       Now, we split the proof in two cases:
       
       \textbf{Case 1}: $\lambda>0$. Exploiting \eqref{g.2}, we have:
\begin{align*}
    c(1+\|u_n\|)&\ge f_\lambda(u_n)-\frac{1}{q}\langle\omega_n,u_n\rangle=\frac{q-2}{2q}\int_\Omega\left[A(x,u_n)\nabla u_n\cdot \nabla u_n+\lambda u_n^2\right]\\
    &\quad-\frac{1}{2q}\int_\Omega \left(D_sA(x,u_n)\nabla u_n\cdot\nabla u_n\right) u_n+\int_\Omega\left[\frac{1}{q}g(x,u_n)u_n-G(x,u_n)\right]\\
    &\ge\frac{q-2-\gamma'}{2q}\int_\Omega\left[ A(x,u_n)\nabla u_n\cdot\nabla u_n+\lambda u_n^2\right]-K'(R,\varepsilon)\\
    &\ge\beta\|u_n\|^2-K'(R,\varepsilon),
\end{align*}
for some $\beta>0$.
Therefore, $\{u_n\}$ is bounded in $H^1(\Omega)$.

\textbf{Case 2}: $\lambda\le0$.
We add and subtract $u_n^2$, obtaining:
\[\int_\Omega\left[A(x,u_n)\nabla u_n\cdot \nabla u_n+ u_n^2\right]-(1-\lambda)\int_\Omega u_n^2.\]
We apply the weighted Young's inequality:
\[\int_\Omega u_n^2\le \frac{p-2}{p}\delta^{-\frac{p}{p-2}}+\frac{2}{p}\delta^{\frac p2}\int_\Omega |u_n|^p.\]
 Thus, for small $\delta > 0$, we get:
\begin{align*}
    c(1+\|u_n\|)&\ge\frac{q-2-\gamma'}{2q}\int_\Omega \left[A(x,u_n)\nabla u_n\cdot\nabla u_n+u_n^2\right]-K''(R,\varepsilon,\delta)\\
    &\ge\frac{q-2-\gamma'}{2q}\min\{\nu,1\}\|u_n\|^2-K''(R,\varepsilon,\delta).
\end{align*}
Hence, $\{u_n\}$ is bounded in $H^1(\Omega).$
   \end{proof} 
\section{Main Theorem}
\begin{proof}[Proof of Theorem \ref{main result1}]
We proceed to verify the assumptions of Theorem \ref{MPequi}, following \cite[Theorem 6.6]{struwe2008variational}. \\
    \textbf{Step 1}: There exist a subspace $W\subset H^1(\Omega)$ of finite codimension and  $\rho,\alpha>0$ such that $f_\lambda(u)\ge\alpha$ for every $u\in \partial B_\rho\cap W$. 
    
     Indeed, let $\varphi_j$ be the eigenfunctions associated to the following elliptic problem
     \[
    \begin{cases}
    -\Delta \varphi_j=\lambda_j\varphi_j&\text{in $\Omega$},\\
    \frac{\partial \varphi_j}{\partial\eta}=0&\text{on $\partial\Omega$},
    \end{cases}
    \]
   with eigenvalue $\lambda_j$ (recall that $\lambda_0=0$ is the first eigenvalue and $\varphi_0=|\Omega|^{-1}$ is the first eigenfunction). 
   
  If $\lambda>0$, we have
    \begin{align*}
        f_\lambda(u)&\ge \frac{\nu}{2}\int_\Omega |\nabla u|^2+\frac{\lambda}{2}\int_\Omega u^2-\int_\Omega |a(x)u|-\frac{|b|}{p}\int_\Omega |u|^p\\
        &\ge\frac{\nu}{2}\int_\Omega |\nabla u|^2+\frac{\lambda}{2}\int_\Omega u^2-c_2\|u\|-\frac{|b|}{p}\int_\Omega|u|^p\\
        &\ge c_1\|u\|_{\lambda}^2-c_3\left(\int_\Omega|{u}|^2\right)^{\frac{t}{2}}\cdot\left(\int_\Omega|{u}|^{2^*}\right)^{\frac{p-t}{2^*}}-c_2\|u\|\\
        &\ge c_1\|u\|_{\lambda}^2-c_3\lambda_{h_0}^{-\frac{t}{2}}\|{u}\|^{p}-c_2\|u\|\\
        &=\|u\|\left[\left(c_1-c_3\lambda_{h_0}^{-\frac{t}{2}}\|{u}\|^{p-2}\right)\|u\|-c_2\right]
    \end{align*}
    with $\rho:=1+c_2$, $\frac{t}{2}+\frac{p-t}{2^*}=1$,  and $h_0$ large enough such that
    \[c_1-c_3\lambda_{h_0}^{-\frac{t}{2}}\rho^{p-2}\ge1.\]
    Thus,    $f_\lambda(u)\ge\alpha$ in $\partial B_\rho\cap W$, where  $W=\overline{\text{span}\{{\varphi}_j,\ j\ge h_0\}}$ and $\alpha>0$ is a positive number.
    
     If $\lambda\le0$ we consider $W=\overline{\text{span}\{\varphi_j,\ j\ge h_0\}}$  and $h_0\in\N$ is such that $\lambda\ge-\nu\lambda_{h_0}$. Let $u\in\partial B_\rho\cap W$.
   \begin{align*}
       f_\lambda(u)&\ge \frac{\nu}{2}\int_\Omega |\nabla u|^2+\frac{\lambda}{2}\int_\Omega u^2-\int_\Omega |a(x)u|-|b|\int_\Omega |u|^p\\
       &\ge\frac{\lambda_{h_0}\nu+\lambda}{2}\|u\|^2-c_2\|u\|-\frac{|b|}{p}\int_\Omega|u|^p\\
       &\ge\frac{\lambda_{h_0}\nu+\lambda}{2}\|u\|^2-c_3\left(\int_\Omega|{u}|^2\right)^{\frac{t}{2}}\cdot\left(\int_\Omega|{u}|^{2^*}\right)^{\frac{p-t}{2^*}}-c_2\|u\|\\
       &\ge\frac{\lambda_{h_0}\nu+\lambda}{2}\|u\|^2-c_3\lambda_{h_0}^{-\frac{t}{2}}\|{u}\|^{p}-c_2\|u\|\\
       &=\|u\|\left[\left(\tilde c_1-c_3\lambda_{h_0}^{-\frac{t}{2}}\|{u}\|^{p-2}\right)\|u\|-c_2\right]
   \end{align*}
 with $\frac{t}{2}+\frac{p-t}{2^*}=1$ and $\tilde c_1=\frac{\lambda_{h_0}\nu+\lambda}{2}$. As before, we can choose  $h_0$ large enough such that  $f_\lambda\ge\alpha$ in $\partial B_\rho\cap W$ with $\alpha>0$.\\
  \textbf{Step 2:} For every subspace $V\subset H^1(\Omega)$ of finite dimension there exists $R>\rho$ such that $f_\lambda\le 0$ in $B_R^c\cap V$.\\ Let $v\in V\cap L^\infty(\Omega)$ with $\|v\|=1$. From \eqref{g.2} we have
    \[\int_\Omega G(x,tv)\ge t^q\int_\Omega a_0(x)|v|^{q}-K,\ \ a_0\in L^1(\Omega),  \ K>0.\]
    Thus,
    \[f_\lambda(tv)\le \frac{NCt^2}{2}\int_\Omega |\nabla v|^2+\frac{\lambda t^2}{2}\int_\Omega v^2-t^{q}\int_\Omega a_0(x)|v|^q-K\to-\infty,\]
    as $t\to+\infty$. Since $H^1(\Omega)\cap L^\infty(\Omega)$ is dense in $H^1(\Omega)$, $f(tw)\to-\infty$ as $t\to+\infty$ for every $w\in V$ with $\|w\|=1$. Indeed, for each $w\in V$ there exists $\{v_n\}\subset V\cap L^\infty(\Omega)$ with $v_n\to w$ in $H^1(\Omega)$ and 
    \[f_\lambda(tw)=\lim_{n\to+\infty}f_\lambda(tv_n)\le0,\]
    for $t>0$ large enough. \\
    \textbf{Step 3}: 
     $f_\lambda$ satisfies also the $(PS)$-condition.
     
    According to Theorem \ref{CPS bounded} and Corollary \ref{CPS condition}, we have that $f_\lambda$ satisfies the $(CPS)$-condition, which implies that $f_\lambda$ satisfies also the $(PS)$-condition by Proposition \ref{PS e CPS}.

   Thus, Theorem \eqref{MPequi} implies that there exists a sequence $\{u_h\}$ of weak solutions of \eqref{P} such that $f_\lambda(u_h)\to+\infty$. Moreover, $u_h\in H^1(\Omega)\cap L^\infty(\Omega)$ by Theorem \ref{regularity}. 
   
   Finally, assume that $u=c\in\R$ is a constant solution of \eqref{P}. 
  From \eqref{g.2} we have that there exist $k_1,k_2>0$ such that
\[
f_\lambda(c)\le\frac{\lambda}{2}|\Omega|-k_1|c|^q-k_2.
\]
Then the energy of constant functions is bounded and we can choose $h_0\in\N$ such that 
\[
f_\lambda(u_h)>f_\lambda(c),\qquad\text{ for every $h\ge h_0$.}
\]
 Thus, the sequence $\{u_h\}_{h\ge h_0}\subset H^1(\Omega)\cap L^\infty(\Omega)$
is a sequence of non-constant weak solutions of \eqref{P} and the proof is complete.
\end{proof}


\end{document}